\documentclass[12pt]{article}

\usepackage{amscd}
\usepackage{latexsym,amsmath,amssymb,amsbsy,amsthm}
\usepackage[square]{natbib}

\usepackage{hyperref}
\hypersetup{
    unicode=false,          % non-Latin characters in Acrobat’s bookmarks
    pdftoolbar=true,        % show Acrobat’s toolbar?
    pdfmenubar=true,        % show Acrobat’s menu?
    pdffitwindow=false,     % window fit to page when opened
    pdfstartview={FitH},    % fits the width of the page to the window
    pdftitle={},    % title
    pdfauthor={Soumendu Sundar Mukherjee},     % author
    pdfsubject={Mathematics},   % subject of the document
    pdfcreator={Creator},   % creator of the document
    pdfproducer={Producer}, % producer of the document
    pdfkeywords={keyword1} {key2} {key3}, % list of keywords
    pdfnewwindow=true,      % links in new window
    colorlinks=true,       % false: boxed links; true: colored links
    linktoc=page,
    linkcolor=blue,          % color of internal links
    citecolor=blue,        % color of links to bibliography
    filecolor=green,      % color of file links
    urlcolor=cyan,           % color of external links
    linkbordercolor={1 1 1},
    citebordercolor={0 1 0},
    urlbordercolor={0 1 1}
}

\usepackage{lscape,fancyhdr,fancybox}
\usepackage{graphicx,epsfig, xy}
\usepackage{color}
\usepackage{amsfonts}
\usepackage[hmarginratio=1:1, vmarginratio =1:1,
textheight=22cm, textwidth=15cm]{geometry}

\usepackage[title,titletoc,toc]{appendix}

\numberwithin{equation}{section}

\DeclareMathOperator{\Var}{Var}
\DeclareMathOperator{\tr}{tr}

\newtheorem{prop}{Proposition}[section]
\theoremstyle{plain}

\theoremstyle{plain}
\newtheorem{theorem}{Theorem}[section]
\theoremstyle{plain}
\newtheorem{lemma}{Lemma}[section]
\theoremstyle{plain}

\theoremstyle{plain}
\newtheorem{exm}{Example}[section]
\theoremstyle{remark}
\newtheorem{remark}{Remark}[section]
\theoremstyle{remark}

\newcommand{\E}{\mathbb{E}}

\begin{document}

\def\shorttitle{Bulk behaviour of Schur-Hadamard products}
 \def\shortauthors{Arup Bose and Soumendu Sundar Mukherjee}
 \fancyhf{}
 \renewcommand\headrulewidth{0pt}
 \fancyhead[C]{% comment out when using twoside option %
 \ifodd\value{page}
   \small\scshape\shortauthors
 \else
   \small\scshape\shorttitle
 \fi
 }
 \fancyhead[L]{
   \ifodd\value{page}
     \thepage
   \fi
   }
 \fancyhead[R]{
   \ifodd\value{page}
   \else
      \thepage
     \fi
     }
  %\fancyhead[CE]{\small\scshape\shorttitle} % uncomment when using twoside option%
 %\fancyhead[CO]{\small\scshape\shortauthors}
 \pagestyle{fancy}
 %\rhead{\thepage}
 %\fancyhead{}
 %\fancyhead[LO]{\leftmark}
 %\fancyhead[RE]{\rightmark}
 %\fancyhead[LE,RO]{\thepage}
 %\renewcommand{\headrulewidth}{0pt}
 %\fancyfoot{}
  
\title{\textbf{\LARGE \sc
Bulk behaviour of Schur-Hadamard products of symmetric random matrices}}
 \author{
 %\hspace{.02\textwidth}
 \parbox[t]{0.50\textwidth}{{\sc Arup Bose}
 \thanks{Research  supported by J.C.~Bose National Fellowship, Dept.~of Science and Technology, Govt.~of India.}
 \\ {\small Statistics and Mathematics Unit\\ Indian Statistical Institute\\ 203 B. T.~Road, Kolkata 700108\\ INDIA\\  E-mail: \texttt{bosearu@gmail.com}\\}}
\parbox[t]{0.5\textwidth}{{\sc Soumendu Sundar Mukherjee}
  \\ {\small Master of Statistics student\\ Indian Statistical Institute\\ 203 B. T.~Road, Kolkata 700108\\ INDIA\\  E-mail: \texttt{soumendu041@gmail.com}\\}}}

\date{First version February 10, 2014\\ This version March 12, 2014}

\maketitle
\begin{abstract}
We develop a general method for establishing the existence of the Limiting Spectral Distributions (LSD) of Schur-Hadamard products of independent symmetric patterned random matrices. We apply this method to show that the LSDs of Schur-Hadamard products of some common patterned matrices exist and identify the limits. In particular, the Schur-Hadamard product of independent Toeplitz and Hankel matrices has the semi-circular LSD. We also prove an invariance theorem that may be used to find the LSD in many examples.
\end{abstract}
   \medskip

\medskip

\bigskip
\bigskip

\section{Introduction}\label{sec:intro}
Let $A_n$ be an  $n\times n$ 
matrix with eigenvalues $\lambda_1, \ldots , \lambda_n$. 
The empirical spectral measure $\mu_n$ of $A_n$ is the
random
measure 
\begin{equation}
\mu_n= \frac{1}{n} \sum_{i=1}^n \delta_{ \lambda_i},
\end{equation}
where $\delta_{x}$ is the Dirac delta measure at $x$. The
corresponding random
probability distribution function  is known as the \emph{Empirical
Spectral Distribution} (ESD) and  denoted by
$F^{A_n}$.
The sequence $\{F^{A_n}\}$ is said to converge (weakly) almost surely to a non-random 
distribution function $F$ if, outside a null set, as $n \to \infty$, $F^{A_n}(\cdot) \rightarrow F(\cdot)$ at all
continuity points of $F$. $F$ is known as the \textit{Limiting Spectral Distribution} (LSD). If the latter convergence is in probability, then the weak convergence is said to hold in probability.

There has been a lot of recent work on obtaining the LSDs  of
large dimensional patterned random matrices. These matrices may be defined as follows. Let $\{ x_i,  x_{i,j}\ i, j \geq 0\}$ be a sequence of random variables, 
 called an  \textit{input sequence}. Let 
 $\mathbb Z$ be the set of all integers and let 
 $\mathbb Z_{+}$ be the set of all non-negative integers. Let
\begin{equation}\label{eq:link}
L_n: \{1, 2, \ldots n\}^2 \to \mathbb Z_{+} \ \text{(or} \ \ \mathbb Z_{+}^2) \  n \geq 1, 
\end{equation} be a sequence of functions. We write $L_n=L$ and call it the \textit{link} function and by abuse of
notation we write $\mathbb Z_{+}^2$ as the common domain of $\{ L_n
\}$. Matrices  of the form
\begin{equation}\label{eq:patternmatrix}
A_n=n^{-1/2}((x_{L(i,j)}))
\end{equation}
are called \textit{patterned random matrices}. 
If
$L(i,j)=L(j,i)$ for all $i,j$,  then the matrix is symmetric. In this article we shall denote the LSD of $\{n^{-1/2}A_n\}$, if it exists, by $\mathcal{L}_{A}$.

There are a host of LSD results for  real symmetric patterned random matrices. 
See, for example, \citet{bose2008another} for a detailed description. The symmetric patterned matrices that have received particular attention in the literature are the Wigner, Toeplitz, Hankel,  Reverse Circulant and the Symmetric Circulant matrices. Their link functions is given in Table~\ref{table1}. 
\begin{table}[htbp]
\centering
\begin{tabular}{|c|c|c|}
\hline
Matrix & Notation & Link function\\
\hline
Wigner & $W_n$ & $L_W(i,j)=(\min\{i,j\},\max\{i,j\})$\\
Toeplitz & $T_n$ &$L_T(i,j)=|i-j|$\\
Hankel & $H_n$ & $L_H(i,j)=i+j$\\
Symmetric Circulant & $SC_n$ & $L_{SC}(i,j)=\frac{n}{2}-|\frac{n}{2}-|i-j||$\\
Reverse Circulant & $RC_n$ & $L_{RC}(i,j)=(i+j)(\text{mod } n)$\\
Doubly Symmetric Hankel & $DH_n$ & $L_{DH}(i,j)=\frac{n}{2}-|\frac{n}{2}-(i+j)(\text{mod } n)|$\\
\hline
\end{tabular}
\vskip10pt
\caption{Some common symmetric patterned matrices and their link functions.}
\label{table1}
\end{table}

LSD existence is also known for the upper triangular versions of these matrices \citep{basu2012spectral}. Joint convergence in terms of convergence of moments of all polynomials of these matrices has also been established in varying degrees (see, for example, \citet{bose2011convergence,basu2012joint}). 

However, the Schur-Hadamard (entrywise) product of such matrices does not seem to have been dealt with in any systematic manner. Such matrices have come up in the random matrix literature in specific situations. For example, \citet{bai2007semicircle} considered the problem of finding the LSD of a sparse sample covariance $S$-matrix. They modeled the sparsity by taking Schur-Hadamard product with a sparse $0$-$1$ Wigner matrix and established the semi-circular law as the LSD of the resulting sparse $S$-matrix under Lindeberg-type conditions on the matrix-entries.

More recently, \citet{beckwith2011distribution} considered Schur-Hadamard product of a $\pm 1$ Bernoulli$(p)$ Wigner matrix with a Toeplitz matrix and established the existence of the LSD. They found that when $p=0.5$, the LSD is the familiar semi-circular law, but when $p\neq 0.5$, the limiting moments are polynomials in $(2p-1)$ whose coefficients could not be identified, and as $p$ approaches $1$ these moments approach the corresponding moments of the LSD of the Toeplitz matrix. \citet{goldmakher2013spectral} considered randomly weighted sequences of $d$-regular graphs with size growing to $\infty$, which amounts to taking Schur-Hadamard product of random real symmetric weight matrices with the adjacency matrices of the graphs, and established the existence of a limiting spectral distribution that depends only on $d$ and the distribution of the weights, under the usual decay condition on the number of $k$-cycles relative to the graph-size, for each $k\geqslant 3$ (in the unweighted case, the limiting spectral distribution is the well-known Kesten's measure).

In this article we shall consider Schur-Hadamard products of real symmetric patterned matrices and establish results on their LSD. In particular, we prove an invariance theorem which yields the result of \citet{beckwith2011distribution}, when $p=0.5$, as a special case. We also consider the Schur-Hadamard product of Toeplitz and Hankel matrices (and other combinations like Toeplitz and Reverse Circulant etc.) and show that the LSD is the semi-circular law. Table~\ref{table2} summarizes our results about the six patterned matrices mentioned in Table~\ref{table1}.
\begin{table}[htbp]
\centering
\begin{tabular}{|c|c|c|}
\hline
$X_n$ & $Y_n$ & LSD\\
\hline
$W_n$ & $T_n$, $H_n$, $SC_n$, $RC_n$, $DH_n$  & $\mathcal{L}_W$\\
$T_n$, $SC_n$ & $H_n$, $RC_n$, $DH_n$ & $\mathcal{L}_W$\\
$T_n$ & $SC_n$ & $\mathcal{L}_T$\\
$H_n$ & $RC_n$, $DH_n$ & $\mathcal{L}_H$\\
$RC_n$ & $DH_n$ & $\mathcal{L}_{RC}$\\
\hline
\end{tabular}
\vskip10pt
\caption{LSDs of several Schur-Hadamard products.}
\label{table2}
\end{table}
\section{Preliminaries}\label{sec:prelims}
We shall use the method of moments to establish the existence of the LSD. For any matrix $A$, let $\beta_h(A)$ denote the $h$-th moment of the ESD of $A$. 
The following lemma, which is easy to prove, will be useful.
\begin{lemma}\label{lem:main}
Let $\{A_n\}$ be a sequence of random matrices with all real eigenvalues. Suppose there exists a sequence $\beta_h$ such that \vskip10pt

\noindent 
$(i)$  For every $h\geq 1$, $\E(\beta_h(A_n))\rightarrow \beta_h$, \vskip10pt

\noindent $(ii)$  $\Var(\beta_h(A_n))\rightarrow 0$ for every $h\geq 1$ and \vskip10pt

\noindent 
$(iii)$ the sequence $\{\beta_h\}$ satisfies Carleman's condition,  $\sum\beta_{2h}^{-1/2h}=\infty$.\vskip10pt

Then the LSD of $F^{A_n}$ exists in probability and equals $F$ with moments $\{\beta_h\}$. If in place of $(ii)$, $\beta_h$ satisfies the stronger condition\vskip10pt

\noindent $(ii')$  $\sum_{n=1}^{\infty}\E[\beta_h(A_n)-\E(\beta_h(A_n))]^4<\infty$ for every $h\geq 1$, \vskip10pt

then the LSD exists in the almost sure sense.
\end{lemma} 

We shall consider three assumptions on the input sequence. \vskip10pt

\noindent {\bf (A1).} The input random variables are independent and uniformly bounded with mean $0$, and variance $1$.\vskip10pt
\noindent {\bf (A2).} The input random variables are i.i.d. with mean $0$ and variance $1$.\vskip10pt
\noindent {\bf (A3).} The input random variables are independent with mean $0$ and variance $1$, and with uniformly bounded moments of all orders.\vskip10pt

We now quickly recollect some terminology and notation from the general theory of patterned matrices
(see \citet{bose2008another}). 

A link function $L$ is said to satisfy \textit{Property B} if 
\[
\Delta_L:=\sup_n\sup_{t\in \text{range}(L)}\sup_{1\leqslant k\leqslant n} \#  \{l \mid 1\leqslant\ l \leqslant n,\,  L(K,l)=t \} <\infty.
\]
In other words, the total number of times any particular variable appears in any row is uniformly bounded. All the matrices introduced so far satisfy this property. For example, $\Delta_{L_W}=1$ and $\Delta_{L_T}=2$. 

Let  $L$ be the link function of the matrix $A_n$. Define
\[
k^A_n:=\#\{L_n(i,j)\, \mid \, 1\leq i, j \leq n\},
\]
and
\[
\alpha^A_n:=\max_k \#\{(i,j)\, \mid \, L_n(i,j)=k\}.
\]
Consider the following conditions on $L$:
\begin{equation}\label{Lcond}
k^A_n  \rightarrow \infty \,\,\, \text{ and }\,\,\, k^A_n\alpha^A_n  =O(n^{2}),
\end{equation}
where $O(\cdot)$ is the Landau big ``Oh" notation: for two real valued functions $f$ and $g$ defined on the set of integers, one writes $f(n)=O(g(n))$ if there is a constant $C$ independent of $n$ such that $|f(n)|\leqslant C |g(n)|$ for all $n>N_0$ for some integer $N_0$. Often the constant $C$ and the ``cut-off" $N_0$ depend on other ``parameters'', say $\alpha$, of the problem at hand, and one makes this explicit by writing $f(n)=O_{\alpha}(g(n))$. For instance, we might have written in (\ref{Lcond}) that $k^A_n\alpha^A_n  =O_L(n^{2})$, because the constant  implied by the big ``Oh'' might depend on the link function. However, from now on we shall suppress these ``parameters'' to avoid notational clutter. 
Note that all the link functions introduced so far satisfy the above conditions. It is known that if the LSD of $\{n^{-1/2}A_n\}$ exists under Assumption (A1), then the same LSD continues to hold under (A2)  or (A3), provided the link function satisfies Property B and  Conditions (\ref{Lcond}). The same continues to be true for Schur-Hadamard products. Thus, in our arguments, without loss of any generality, we assume that (A1) holds. Traditionally, LSD results are stated under (A1),  and (A3) is appropriate while studying the joint convergence of more than one sequence of matrices. 

The \emph{Moment-Trace Formula} plays a key role in this approach. A function $$\pi : \{0, 1, \cdots, h\} \rightarrow \{1,2,\cdots,n\}$$ with $\pi(0) = \pi(h)$ is  called a \emph{circuit} 
of
length $h$. The dependence of a circuit on $h$ and $n$ is suppressed. Then
\begin{equation}
\beta_h(A)= \frac{1}{n}\tr(A^h)=\frac{1}{n}\sum_{\pi \text{ circuit of length $h$}}a_\pi,
\end{equation}
where 
\[
a_{\pi} := a_{L(\pi(0),\pi(1))}a_{L(\pi(1),\pi(2))}\ldots a_{L(\pi(h-1),\pi(h))}.
\]

If $L(\pi(i-1),\pi(i))=L(\pi(j-1),\pi(j))$, with $i<j$, we shall use the notation $(i,j)$ to denote such a match of the $L$-values. Also if $L$ is the link function of $A_n$, we will often use the further shorthand notation $i\sim^A j$ in lieu of the phrase ``$(i,j)$ is an $L$-match''. If an $L$-value is repeated exactly $e$ times, we say that the circuit $\pi$ has an edge (or $L$-edge) of order $e$ ($1 \leqslant e \leqslant h$). If
$\pi$ has all $e \geqslant 2$, then it is called $L$-matched (in short matched). If $\pi$ has only order two edges, then it is called pair-matched. From the general theory, it follows that only pair-matched circuits are relevant when computing limits of moments. 

To deal with Conditions $(ii)$ and $(ii')$ of Lemma~\ref{lem:main}, we need multiple circuits: $t$ circuits $\pi_1, \pi_2,$ $\cdots,\pi_t$ are \emph{jointly $L$-matched} if each $L$-value occurs at least twice across all circuits. They are \emph{across $L$-matched} if each circuit has at least one $L$-value which occurs in at least one of the other circuits. 
 
Two circuits $\pi_1$ and $\pi_2$ are equivalent if and only if their $L$-values respectively match at the
same locations, i.e., if for all $i, j$,
\[
L(\pi_1(i-1),\pi_1(i))=L(\pi_1(j-1),\pi_1(j)) \Leftrightarrow L(\pi_2(i-1),\pi_2(i))=L(\pi_2(j-1),\pi_2(j)).
\]

Any equivalence class can be indexed by a partition of $\{1,2,\cdots, h\}$. We label these partitions by \emph{words} of length $h$ of letters where the first occurrence of each letter is in alphabetical order. For example, if $h = 4$ then the partition
$\{\{1,3\},\{2,4\}\}$ is represented by the word $abab$. This identifies all circuits $\pi$ for
which $L(\pi(0),\pi(1)) = L(\pi(2), \pi(3))$ and $L(\pi(1),\pi(2)) = L(\pi(3),\pi(1))$.
Let $w[i]$ denote the $i$-th entry of $w$. The equivalence class corresponding to $w$ is 
\[
\Pi(w) := \{\pi \mid w[i]=w[j]\Leftrightarrow L(\pi(i-1),\pi(i))=L(\pi(j-1),\pi(j))\}.
\]
Note that the number of partition blocks corresponding to $w$ is same as the number of distinct letters in $w$, which we denote by $|w|$. 
By varying $w$, we obtain all the equivalence classes. It is important to note that for any
fixed $h$, even as $n\rightarrow \infty$, the number of words remains finite but
the number of circuits in any given $\Pi(w)$ may grow indefinitely. Henceforth, we shall denote the set of all words of length $h$ by $\mathcal{A}_{h}$. 
Notion of matches carry over to words and we shall again use the notation $(i,j)$, $i<j$ to denote the match $w[i]=w[j]$ in $w$. Note that a word is \emph{pair-matched} if every letter 
appears exactly twice in that word. The set of all pair-matched words of length $2k$
is  denoted by $\mathcal{W}_{2k}$.
For technical reasons it is often easier to deal with a class larger than $\Pi(w)$:
\[
\Pi^*(w) = \{\pi \mid w[i]=w[j]\Rightarrow L(\pi(i-1),\pi(i))=L(\pi(j-1),\pi(j))\}.
\]

Any $i$ (or $\pi(i)$ by abuse of notation) is a \emph{vertex}. It is \emph{generating} if either $i = 0$
or $w[i]$ is the first occurrence of a letter. Otherwise, it is called non-generating. For
example, if $w = abbcab$ then $\pi(0), \pi(1), \pi(2), \pi(4)$ are generating and $\pi(3), \pi(5), \pi(6)$
are non-generating. 
By Property B a circuit is completely determined, up to finitely many choices, by its generating vertices. The number of generating vertices in any circuit in $\Pi(w)$ is $|w| + 1$
and hence
\[
\#\Pi(w) \leqslant \#\Pi^*(w)= O(n^{|w|+1}).
\]
The set of generating vertices (indices) is denoted by $S$. The dependence on the word $w$ will, in general, be clear from the context. Sometimes we shall write $\Pi_A(w)$, $\Pi^*_A(w)$ or $S_A$ to emphasise dependence on the matrix $A_n$.  

From the general theory, it follows that for a sequence of patterned random matrices $\{n^{-1/2}A_n\}$ the LSD exists if for all  $w \in \mathcal{W}_{2k}$, the following limit exists:
\begin{equation}\label{wordlimit}
p(w)=\lim n^{-(1+k)}\# \Pi(w)=\lim n^{-(1+k)}\# \Pi^*(w)
\end{equation}
and in that case the $2k$-th moment of the LSD is given by
$$\beta_{2k}=\sum_{w\in \mathcal{W}_{2k}} p(w).$$
The existence of the LSD for the Wigner, Hankel, Toeplitz, Reverse Circulant, Symmetric Circulant and Doubly Symmetric Hankel matrices can be established by verifying that for every $k$, $p(w)$ exists for every pair-matched word (see \citet{bose2008another,bose2010patterned}). 

In the next section we extend the above approach for a single sequence to the Schur-Hadamard product of two sequences. It may be noted that although we work with only two sequences, it is quite straightforward to extend the results to any finite number of sequences. 

\section{Schur-Hadamard product}\label{sec:shproduct} 
Suppose  $\{X_n\}$ and $\{Y_n\}$ are two independent sequences of patterned symmetric random matrices with link functions  $L_X$ and $L_Y$ respectively and their input sequences satisfy (A1). Let $Z_n=X_n\odot Y_n $ be their Schur-Hadamard product. 
We shall employ the moment method via the word approach to study the LSD of 
$\{n^{-1/2}Z_n\}$. Note that $Z_n$ is not necessarily a patterned matrix. However, 
many of the arguments of the general theory for a single matrix may be used for $Z_n$ with appropriate modifications. 
 
Define
\[
k_n^Z :=\#\{(L_X(i,j),L_Y(i,j))\, \mid \, 1\leq i, j \leq n\},
\]
i.e., $k_n^Z$ is the total number of the $X$ and $Y$ variable pairs appearing in the matrix $Z_n$. Also let
\[
\alpha_n^Z :=\max_{(k,l)}\#\{(i,j)\, \mid \, (L_X(i,j), L_Y(i,j))=(k,l)\},
\]
i.e., $\alpha_n^Z$ is the maximum number of occurrences of any $X$ and $Y$ variable pair in $Z_n$.

It is easy to show that if $L_X$ and $L_Y$ both satisfy 
(\ref{Lcond}), then 
$\{k_n^Z, \alpha_n^Z\}$ also satisfy these conditions.

To elaborate, note that
\[
\max\{k_n^X,k_n^Y\} \leqslant k_n^Z \leqslant k_n^X + k_n^Y.
\]
So $k_n^Z\rightarrow \infty$.

Similarly, it is obvious that
\[
 \alpha_n^Z \leqslant \min\{\alpha_n^X,\alpha_n^Y\}.
\]
Therefore
\begin{align*}
k_n^Z \alpha_n^Z & \leqslant (k_n^X + k_n^Y) \min\{\alpha_n^X,\alpha_n^Y\}\\
                             & \leqslant \alpha_n^X k_n^X + \alpha_n^Y k_n^Y\\
                             & = O(n^2)+O(n^2)=O(n^2),
\end{align*}
which completes the verification of (\ref{Lcond}).

For the time being assume \textit{all moments exist}. Then the \emph{moment trace formula} for $n^{-1/2}Z_n$ becomes
\begin{align*}
\beta_h(F^{n^{-1/2}Z_n})& =n^{-(1+\frac{h}{2})}\sum_{\pi:\pi \text{  circuit of length $h$}}z_{\pi}\\
                        & =n^{-(1+\frac{h}{2})}\sum_{\pi:\pi \text{  circuit of length $h$}}x_{\pi}y_{\pi}.
\end{align*} 
Therefore
\begin{align*}
\E(\beta_h(F^{n^{-1/2}Z_n}))& = n^{-(1+\frac{h}{2})}\sum_{\pi:\, \pi \text{ circuit of length $h$}}\E(x_{\pi}y_{\pi}) \\
& = n^{-(1+\frac{h}{2})}\sum_{\pi:\, \pi \text{ circuit of length $h$}}\E x_{\pi}\E y_{\pi}.
\end{align*}

The two (possibly different) link functions $L_X$ and $L_Y$ induce two partitions (via words) $\mathcal{P}_X$ and $\mathcal{P}_Y$ on the set of all circuits of length $h$. Consider the resultant $\mathcal{P}_Z$ of these two partitions defined as
\[
\mathcal{P}_Z=\{\Pi_X(w)\cap \Pi_Y(w') \, \mid \, w,w'\in \mathcal{A}_{h} \}.
\]
 For the sake of brevity, let us define 
 \[
 \Pi_Z(w,w'):=\Pi_X(w)\cap \Pi_Y(w').
 \]
  Then we can write 
\begin{equation}\label{eq:emt}
\E(\beta_h(F^{n^{-1/2}Z_n})) = n^{-(1+\frac{h}{2})}\sum_{(w,w')\in \mathcal{A}_{h}^2}\ \sum_{\pi \in \Pi_Z(w,w')}\E x_{\pi}\E y_{\pi}.
\end{equation}
We can now state and prove our first lemma. 
\begin{lemma}\label{lem:three_or_more}
 Suppose the input sequences satisfy Assumption (A1) and the link functions $L_X$ and $L_Y$ satisfy Property B. Then circuits which have at least one edge of order $\geqslant 3$ contribute zero to the possible limit of  $\E(\beta_h(F^{n^{-1/2}Z_n}))$. As a consequence, for every odd $h$, $\E(\beta_h(F^{n^{-1/2}Z_n}))\to 0$ and when $h$ is even, only those words $(w, w^{\prime})$ where both are pair-matched can contribute to the possible limit of moments. 
\end{lemma}
\begin{proof}
First note that if a circuit $\pi$ is not $L_X$-matched or $L_Y$-matched then $\E x_{\pi}\E y_{\pi}=0$ and consequently such a $\pi$ does not have any contribution to the moment. We thus need to consider only matched circuits (or words) henceforth. Denote by $C^{L}_{h,3+}$ the set of all $L$-matched circuits of length $h$  with at least one edge of order $\geqslant 3$. If $L$ satisfies Property B, then Lemma 1(a) of \citet{bose2008another} ensures that
\begin{equation}\label{eq:three_or_more}
\#C^{L}_{h,3+}= O(n^{\lfloor(h+1)/2\rfloor}),
\end{equation}
where $\lfloor x\rfloor$ denotes the largest integer contained in $x$. Using this in our context we have
\begin{align*}
\#(C^{L_X}_{h,3+}\cup C^{L_Y}_{h,3+}) & \leq  \#C^{L_X}_{h,3+} + \#C^{L_Y}_{h,3+} \\
									& = O(n^{\lfloor(h+1)/2\rfloor}) +O( n^{\lfloor(h+1)/2\rfloor}) \\
									& = O( n^{\lfloor(h+1)/2\rfloor}).
\end{align*}

This implies that the circuits which have at least one $L_X$-edge or $L_Y$-edge of order $\geqslant 3$ do not contribute in the limit. It follows immediately that for odd $h$ the limit of the expected $h$-th moment is zero and for even $h$ only the circuits in $\Pi_Z(w,w')$ where both $w$ and $w'$ are pair-matched can have a potential contribution. 
This proves the lemma.
\end{proof}
\begin{lemma}\label{lem:inprob}
Suppose $L_X$ and $L_Y$ satisfy Property B and the input sequences satisfy Assumption (A1). Then  $\{\beta_h(n^{-1/2}Z_n)\}$ satisfies Condition $(ii)$ of Lemma~\ref{lem:main} for any $h$. 
\end{lemma}
\begin{proof} 
Let $K^L_{h,t}$ be the set of $t$-tuples of circuits $(\pi_1,\cdots,\pi_t)$ of length $h$ such that they are jointly and across $L$-matched. 
Using Lemma~\ref{lem:count1} we have
\begin{align}\label{bound:pairs}
\#(K^{L_X}_{h,2}\cup K^{L_Y}_{h,2}) & \leq  \#K^{L_X}_{h,2} + \#K^{L_Y}_{h,2} \\
\nonumber									& =  O(n^{h+1}) +O(n^{h+1}) \\
\nonumber									& = O(n^{h+1}).
\end{align}
Now write
\begin{align}\label{repr:var}
\Var(\beta_h(n^{-1/2}Z_n))&=\E[\beta_h(n^{-1/2}Z_n)-\E(\beta_h(n^{-1/2}Z_n))]^2\\
\nonumber &= \E[n^{-1}\tr(n^{-1/2}Z_n)^h-\E(n^{-1}\tr(n^{-1/2}Z_n)^h)]^2\\
\nonumber &= \frac{1}{n^{h+2}}\E[\tr Z_n^h-\E\tr Z_n^h]^2\\
\nonumber &=\frac{1}{n^{h+2}}\sum_{(\pi_1,\pi_2)}\E(z_{\pi_1}-\E z_{\pi_1})(z_{\pi_2}-\E z_{\pi_2}),
\end{align}
and decompose
\begin{align}\label{decomp}
\nonumber z_{\pi_j}-\E z_{\pi_j} & = x_{\pi_j}y_{\pi_j}-\E x_{\pi_j}\E y_{\pi_j} \\
							   & = (x_{\pi_j}-\E x_{\pi_j})y_{\pi_j}+(y_{\pi_j}-\E y_{\pi_j})\E x_{\pi_j}.
\end{align}							   
Note that by decomposition (\ref{decomp}) we have
\begin{align}
\label{repr:2}&\E(z_{\pi_1}-\E z_{\pi_1})(z_{\pi_2}-\E z_{\pi_2})\\
\nonumber=& \E ((x_{\pi_1}-\E x_{\pi_1})y_{\pi_1}+(y_{\pi_1}-\E y_{\pi_1})\E x_{\pi_1})((x_{\pi_2}-\E x_{\pi_2})y_{\pi_2}+(y_{\pi_2}-\E y_{\pi_2})\E x_{\pi_2})\\
\nonumber=& \E (x_{\pi_1}-\E x_{\pi_1})(x_{\pi_2}-\E x_{\pi_2}) \E y_{\pi_1}y_{\pi_2} + \E (x_{\pi_1}-\E x_{\pi_1})\E x_{\pi_2}\E y_{\pi_1}(y_{\pi_2}-\E y_{\pi_2})\\
\nonumber& + \E x_{\pi_1}\E (x_{\pi_2} -\E x_{\pi_2})\E (y_{\pi_1}-\E y_{\pi_1})y_{\pi_2} + \E x_{\pi_1}\E x_{\pi_2}\E (y_{\pi_1}-\E y_{\pi_1})(y_{\pi_2}-\E y_{\pi_2})\\
\nonumber=& \E (x_{\pi_1}-\E x_{\pi_1})(x_{\pi_2}-\E x_{\pi_2}) \E y_{\pi_1}y_{\pi_2} + \E x_{\pi_1}\E x_{\pi_2}\E (y_{\pi_1}-\E y_{\pi_1})(y_{\pi_2}-\E y_{\pi_2}).
\end{align}
If $(\pi_1 , \pi_2)$ are not jointly $L_X$-matched, then one of the circuits, say $\pi_1$ , has an
$L_X$-value which does not occur anywhere else. Therefore $\E x_{\pi_1}=0$. So, from (\ref{repr:2}) it follows that
\[
\E(z_{\pi_1}-\E z_{\pi_1})(z_{\pi_2}-\E z_{\pi_2}) = \E x_{\pi_1}(x_{\pi_2}-\E x_{\pi_2}) \E y_{\pi_1}y_{\pi_2}.
\]
But since the input $X$-variable corresponding to the single $L_X$-value appears in the product $x_{\pi_1}(x_{\pi_2}-\E x_{\pi_2})$ inside $x_{\pi_1}$ and is independent of every other term in the product, we conclude that
\[
\E x_{\pi_1}(x_{\pi_2}-\E x_{\pi_2})=0,
\]
and as a result
\[
\E(z_{\pi_1}-\E z_{\pi_1})(z_{\pi_2}-\E z_{\pi_2})=0.
\]
 Therefore, it is enough to consider those  $(\pi_1 , \pi_2)$ which are jointly $L_X$-matched and jointly $L_Y$-matched.

Now suppose that $(\pi_1 , \pi_2)$ are jointly $L_X$ as well as $L_Y$-matched but neither across $L_X$-matched nor across $L_Y$-matched. Then there is a circuit, say $\pi_k$, which is only \emph{self $L_X$-matched}, i.e., none of its $L_X$-values is shared with those of the other circuit. Similarly, there is a circuit $\pi_l$ that is only self $L_Y$-matched. Now note that $(x_{\pi_k}-\E x_{\pi_k})$ is independent of $(x_{\pi_j}-\E x_{\pi_j})$ for $j\neq k$ and similarly $(y_{\pi_l}-\E y_{\pi_l})$ is independent of $(y_{\pi_j}-\E y_{\pi_j})$ for $j\neq l$. Using this in (\ref{repr:2}) we can write
\begin{align*}
\E(z_{\pi_1}-\E z_{\pi_1})(z_{\pi_2}-\E z_{\pi_2})= \E (x_{\pi_1}-\E x_{\pi_1})&\E(x_{\pi_2}-\E x_{\pi_2})\E y_{\pi_1}y_{\pi_2} \\&+ \E x_{\pi_1}\E x_{\pi_2}\E (y_{\pi_1}-\E y_{\pi_1})\E(y_{\pi_2}-\E y_{\pi_2})\\
&=0.
\end{align*}
It thus follows that if $(\pi_1, \pi_2)\notin K^{L_X}_{h,2}\cup K^{L_Y}_{h,2}$, then  
\[
\E(z_{\pi_1}-\E z_{\pi_1})(z_{\pi_2}-\E z_{\pi_2})=0.
\]
Now, because of Assumption (A1), $\E(z_{\pi_1}-\E z_{\pi_1})(z_{\pi_2}-\E z_{\pi_2})$ is bounded uniformly across all possible pairs of circuits and thus, from (\ref{repr:var}) and the bound (\ref{bound:pairs}), we conclude that
\[
\Var(\beta_h(n^{-1/2}Z_n))=O\left(\frac{n^{h+1}}{n^{h+2}}\right)=O(n^{-1}).
\]
This completes the proof.
\end{proof}
\begin{remark}\label{rem:asinappendix}
The verification of Condition $(ii')$ of Lemma~\ref{lem:main} requires more subtle combinatorial analysis. We have included it in Appendix~\ref{sec:asconv}.
\end{remark}
We now show how Carleman's condition can be checked easily under Property B provided that the
word limit exists for every pair-matched word $(w, w^\prime)$.  
Let
\[
\Pi^*_L(w)=\{\pi \, \mid \, w[i]=w[j] \Rightarrow L(\pi(i-1),\pi(i))=L(\pi(j-1),\pi(j)) \text{, for all indices $(i,j)$} \}.
\]
Clearly, $\Pi_L(w)\subseteq \Pi^*_L(w)$. Define
\[
\Pi^*_Z(w,w'):=\Pi^*_X(w)\cap \Pi^*_Y(w').
\]
This also satisfies $\Pi_Z(w,w')\subseteq \Pi^*_Z(w,w')$ and the set $\Pi^*_Z(w,w')\setminus \Pi_Z(w,w')$ is contained in $C^{L_X}_{h,3+}\cup C^{L_Y}_{h,3+}$ which means by Lemma~\ref{lem:three_or_more} that
\[
\lim_n n^{-(1+\frac{h}{2})} \#(\Pi^*_Z(w,w')\setminus \Pi_Z(w,w'))=0 \text{, for each $(w,w')\in \mathcal{W}_{2k}^2$.}
\]
Therefore, if $\lim_n n^{-(1+\frac{h}{2})} \#\Pi_Z(w,w')$ exists, we have
\[
\lim_n n^{-(1+\frac{h}{2})} \#\Pi_Z(w,w')=\lim_n n^{-(1+\frac{h}{2})} \#\Pi^*_Z(w,w')= p_Z(w,w'), \ \ \text{say}.\]
\begin{theorem}\label{thm:LSD} 
Suppose that $L_X$ and $L_Y$ satisfy Property B and the input sequences satisfy Assumption (A1). If the limit $p_Z(w,w^\prime)$ exists for every pair of pair-matched words $(w, w^{\prime})$ then Condition $(i)$ of Lemma~\ref{lem:main} holds and the limit moments satisfy Condition $(iii)$ of Lemma~\ref{lem:main} and hence are the moments of the LSD. The limit law in that case is sub-Gaussian.  If the input sequences satisfy Assumptions (A2) or (A3) and the link functions satisfy Conditions \ref{Lcond}, then the 
same LSD continues to hold. 
\end{theorem} 
\begin{proof}
By the developments so far, 
\begin{align*}
\lim_n \E(\beta_{2k}(F^{n^{-1/2}Z_n})) = &\sum_{(w,w')\in \mathcal{W}_{2k}^2} \lim_n n^{-(1+\frac{h}{2})} \#\Pi_Z(w,w')\\
= &\sum_{(w,w')\in \mathcal{W}_{2k}^2} p_Z(w,w').
\end{align*}
Then the moments of the LSD  would be given by
\begin{equation}
\beta_h=\begin{cases}
0, & \text{ if $h$ is odd} \\
\displaystyle\sum_{(w,w')\in \mathcal{W}_{2k}^2} p_Z(w,w'), & \text{ if $h=2k$.}
\end{cases}
\end{equation}
Now note that for any $w\in \mathcal{A}_{2k}$,
\[
\Pi_X(w)=\bigcup_{w'\in \mathcal{A}_{2k}}\Pi_Z(w,w').
\]
Therefore, for $w\in\mathcal{W}_{2k}$ we have
\[
\bigcup_{w'\in \mathcal{W}_{2k}}\Pi_Z(w,w') \subseteq  \Pi_X(w),
\]
which implies that 
\[
\#\left(\bigcup_{w'\in \mathcal{W}_{2k}}\Pi_Z(w,w')\right)\leq \#\Pi_X(w).
\]
This means that for $(w,w')\in \mathcal{W}_{2k}^2$ we have
\[
\sum_{w'\in \mathcal{W}_{2k}} p_Z(w,w') \leq p_X(w),
\]
and therefore
\[
\sum_{(w,w')\in \mathcal{W}_{2k}^2} p_Z(w,w') \leq \sum_{w \in \mathcal{W}_{2k}} p_X(w).
\]
Since $X$ and $Y$ play a symmetric role in $Z$, we have furthermore
\begin{equation}
\sum_{(w,w')\in \mathcal{W}_{2k}^2} p_Z(w,w') \leq \min\{\sum_{w \in \mathcal{W}_{2k}} p_X(w), \sum_{w' \in \mathcal{W}_{2k}} p_Y(w') \}.
\end{equation}
Recall that the matrices are assumed to satisfy Property B. Also note that the number of pair-matched words of length $2k$ equals 
$\frac{(2k)!}{2^k k!}$. Then it is easy to see that 
\begin{equation}
\beta_{2k} \leq \frac{(2k)!}{2^k k!}\Delta^{k},
\end{equation}
where $\Delta:=\min\{\Delta_{L_X}, \Delta_{L_Y}\}$. This guarantees that $\{\beta_h\}$ satisfies Carleman's condition and the limit law is sub-Gaussian.

Suppose now that the input sequences satisfy Assumption (A2) and the link  functions satisfy Conditions \ref{Lcond}. Then, by appropriate truncation of the input variables and strong law of large numbers, one can reduce that case to the case where Assumption (A1) holds. We omit the tedious details which are similar to the case for a single matrix (see, for example, \citet{bose2008another}).

If the input sequences satisfy Assumption (A3), then all moments are bounded and all the moment calculations and bounds used so far go through. Again, we omit the details. Finally, note that by Remark~\ref{rem:asinappendix} the LSD exists in the almost sure sense. This completes the proof. 
\end{proof}

\section{Some general results}\label{sec:genresults}
Note that by Theorem~\ref{thm:LSD}, the LSD will exist if the limit $p(w, w^\prime)$ exists for each pair-matched word pair $(w,w')$. In this section we shall consider several types of $\{X_n\}$, $\{Y_n\}$ and establish general results on the LSD of $n^{-1/2}Z_n$. We assume that all input sequences satisfy Assumption (A1), (A2) or (A3). But as discussed, we can work under Assumption (A1). We first establish an invariance theorem.
\begin{theorem}\label{thm:gen}
Suppose that $L_X$ satisfies Property B and Conditions (\ref{Lcond}). Suppose  $p_X(w)$ exists for each word so that $\mathcal{L}_X$ exists. Also suppose that there is a transformation $\rho$ such that $L_Y=\rho\circ L_X$ and $L_Y$ also satisfies the above conditions. Then the LSD of $\{n^{-1/2}Z_n\}$ exists almost surely and equals $\mathcal{L}_X$.
\end{theorem}
\begin{proof}
It is enough to show that $\{n^{-1/2}Z_n\}$ has the same limiting moment sequence as $\{n^{-1/2}X_n\}$ and by the theory developed earlier it suffices to look at the even moments only.
Suppose that $w\in \mathcal{W}_{2k}$ and $\pi\in\Pi^*_X(w)$. Then
\begin{align*}
w[i]=w[j] & \Rightarrow L_X(\pi(i-1),\pi(i))=L_X(\pi(j-1),\pi(j)) \\
		  & \Rightarrow \rho \circ L_X(\pi(i-1),\pi(i))=\rho \circ L_X(\pi(j-1),\pi(j)),\\
		  & \text{i.e., } L_Y(\pi(i-1),\pi(i))=L_Y(\pi(j-1),\pi(j)).
\end{align*}
Therefore $\pi \in \Pi^*_Y(w)$ which means that $\Pi^*_X(w)\subseteq \Pi^*_Y(w)$. Therefore
\[
\Pi^*_X(w)\cap \Pi^*_Y(w)=\Pi^*_X(w).
\]
This implies that $p_Z(w,w)=p_X(w)$. But then
\[
p_X(w)=p_Z(w,w)\leq \sum_{w'\in \mathcal{W}_{2k}} p_Z(w,w')\leq p_X(w),
\]
i.e., for each $w \in \mathcal{W}_{2k}$,
\[
\sum_{w'\in \mathcal{W}_{2k}} p_Z(w,w')=p_X(w).
\]
Therefore, one has
\begin{align*}
\beta_{2k}^Z=\lim_n \E(\beta_{2k}(F^{n^{-1/2}Z_n})) & = \sum_{(w,w')\in \mathcal{W}_{2k}\times \mathcal{W}_{2k}} p_Z(w,w') \\
&= \sum_{w\in \mathcal{W}_{2k}}\sum_{w' \in \mathcal{W}_{2k}} p_Z(w,w') \\
&= \sum_{w\in \mathcal{W}_{2k}} p_X(w) \\
&= \beta_{2k}^X.
\end{align*}
This completes the proof.
\end{proof}
\begin{remark}
It is clear from the proof of Theorem~\ref{thm:gen} that we may let $\rho$ depend on $n$. Indeed, all our arguments are for a fixed $n$. For the sake of brevity, we shall continue using $\rho$ instead of $\rho_n$.
\end{remark}
\begin{exm}\label{exm:wigner}
Suppose $X_n=W_n$ and $L_Y$ satisfies Property B and Conditions 
(\ref{Lcond}). Then the LSD of $\{n^{-1/2}W_n \odot Y_n\}$ is the semi-circular law $\mathcal{L}_W$ almost surely. In particular, $Y_n$ could be any one among $T_n$, $H_n$, $RC_n$, $SC_n$ or $DH_n$.

One interesting case is when $X_n$ is a $\pm1$ Bernoulli (with $p=0.5$) Wigner matrix. In that case the Schur-Hadamard product may be interpreted as a randomly censored patterned matrix. Thus, for example, a randomly ($\pm1$) censored Toeplitz matrix will have the semi-circular law as its LSD. This is the result of \citet{beckwith2011distribution} in the $p=0.5$ case.
\end{exm}

\begin{exm}\label{exm:toeplitz}
Suppose that $X_n=T_n$. Also suppose that
\[
L_Y(i,j)=\rho(|i-j|) \text{, for all $(i,j)$}
\]
for some function $\rho$ and 
satisfies Property B and Conditions 
(\ref{Lcond}). 
Then the LSD of $\{n^{-1/2}T_n \odot Y_n\}$ is $\mathcal{L}_{T}$.

In particular, if $Y_n=SC_n$, then the LSD of $\{n^{-1/2}T_n\odot SC_n\}$ is $\mathcal{L}_{T}$.
\end{exm}
\begin{exm}\label{exm:hankel}
Suppose that $X_n=H_n$. Also suppose that 
\[
L_Y(i,j)=\rho(i+j) \text{, for all $(i,j)$}
\]
for some function $\rho$ and satisfies Property B and Conditions 
(\ref{Lcond}). 
Then the LSD of $\{n^{-1/2}H_n\odot Y_n\}$ is $\mathcal{L}_{H}$.

The link functions of the Reverse Circulant and the Doubly Symmetric Hankel matrices have the form $\rho(i+j)$. Therefore, if we take $Y_n=RC_n$ or $DH_n$, we can conclude that the LSD of $\{n^{-1/2}H_n\odot Y_n\}$
is $\mathcal{L}_{H}$.

Similarly, since $L_{DH}$ is of the form $\rho(L_{RC})$, we conclude that the LSD of $\{n^{-1/2}RC_n\odot DH_n\}$ is $\mathcal{L}_{RC}$.
\end{exm}

\begin{remark}
Note that Theorem~\ref{thm:gen} gives us all the rows of Table~\ref{table2} except the second one.
\end{remark}

The next natural question is what happens when we transform both the link functions $L_X$ and $L_Y$. We first consider the case where we have a single sequence $X_n$ and its link function is transformed.

\begin{prop}\label{prop:injective}
Suppose the link function $L_X$ satisfies Property B and Conditions (\ref{Lcond}). Suppose also that $p_X(w)$ exists for each word so that $\mathcal{L}_X$ exists. Suppose $L_Y:=\rho \circ L_X$ where $\rho$ is an injective transformation. 
Then the LSD of $\{n^{-1/2}Y_n\}$ exists and equals $\mathcal{L}_X$.
\end{prop}
\begin{proof}
We first observe that $L_Y$ also satisfies Property B.
To prove this, for $1\leqslant k \leqslant n$ and $t \in \text{range}(L_X)$, define $D^X_{k,t}=\{l\mid 1\leqslant l \leqslant n, \, L_X(k,l)=t\}$ and similarly define $D^Y_{k,t}$. Note that 
\[
L_X(k,l)=t \Leftrightarrow L_Y(k,l)=\rho(t).
\]
Thus $D^X_{k,t}=D^Y_{k,\rho(t)}$. Since $\rho$ is injective and by definition range($L_Y$)$=\rho(\text{range}(L_X))$, we have
\begin{align*}
\Delta_{L_X} & =
\sup_n\sup_{t\in \text{range}(L_X)}\sup_{1\leqslant k\leqslant n} \# D^X_{k,t}\\ 
&=\sup_n\sup_{\rho(t)\in \text{range}(L_Y)}\sup_{1\leqslant k\leqslant n} \# D^Y_{k,\rho(t)}.\\
			& =\Delta_{L_Y}.
\end{align*}
It now suffices to show that $\Pi^*_X(w)=\Pi^*_Y(w)$ for each $w \in \mathcal{W}_{2k}$. From the proof of Theorem~\ref{thm:gen}, it follows that $\Pi^*_X(w)\subseteq \Pi^*_Y(w)$. To show the other way, suppose that $\pi \in \Pi^*_Y(w)$. Then
\begin{align*}
& w[i]=w[j]\\
\Rightarrow & L_Y(\pi(i-1),\pi(i))=L_Y(\pi(j-1),\pi(j))\\
\Rightarrow & \rho \circ L_X(\pi(i-1),\pi(i))=\rho \circ L_X(\pi(j-1),\pi(j))\\
\Rightarrow & L_X(\pi(i-1),\pi(i))=L_X(\pi(j-1),\pi(j)) \,  \text{ (by injectivity of $\rho$)}.
\end{align*}
Thus $\pi \in \Pi^*_X(w)$. Finally, we note that because of the injectivity of $\rho$, we have $k_n^X=k_n^Y$ and $\alpha_n^X=\alpha_n^Y$, so that $L_Y$ satisfies Conditions (\ref{Lcond}). The proof is now complete.
\end{proof}
\begin{exm}\label{exm:injec1}
Take $\rho(i,j)=a^ib^j$, where $a$ and $b$ are coprime positive integers, which is injective and compose it with the Wigner link function to obtain the link function $L(i,j)=a^{i\wedge j}b^{i\vee j}$. Then the LSD of the corresponding patterned random matrix is the  semi-circular law. Similarly, the patterned random matrix with the link function $L(i,j)=(i-j)^2$ has the same LSD as the Toeplitz matrix.
\end{exm}

The following proposition shows that if both $L_X$ and $L_Y$ are transformed via injective maps, then the LSD of their Schur-Hadamard product is preserved.
\begin{prop}\label{prop:both_injective}
Suppose $X_n$ and $Y_n$ are independent patterned matrices where the link functions $L_X$ and $L_Y$ satisfy Property B and Conditions \ref{Lcond}. Suppose $p_Z(w,w')$ exists for each pair-matched word-pair $(w,w')$ so that the LSD of $\{n^{-1/2}X_n\odot Y_n\}$ exists. Suppose $\rho_1$ and $\rho_2$ are injective transformations and $L_U=\rho_1 \circ L_X$ and $L_V=\rho_2 \circ L_Y$. If $U_n$ and $V_n$ are independent patterned matrices with link functions $L_U$ and $L_V$ respectively, then the LSD of $\{n^{-1/2}U_n\odot V_n\}$ exists and is same as that of $\{n^{-1/2}X_n\odot Y_n\}$.
\end{prop}
\begin{proof}
From the proof of Proposition~\ref{prop:injective} we see that both $L_U$ and $L_V$ satisfy Property B and Conditions \ref{Lcond} and we have $\Pi^*_{X}(w)=\Pi^*_U(w)$ and $\Pi^*_{Y}(w)=\Pi^*_V(w)$ for each pair matched word $w$. Then, for each word-pair $(w,w')$, one has $\Pi^*_{X}(w)\cap \Pi^*_{Y}(w')=\Pi^*_U(w)\cap \Pi^*_V(w')$. This completes the proof.
\end{proof}
\begin{exm}
In view of Examples \ref{exm:wigner} and \ref{exm:injec1} we conclude from Proposition \ref{prop:both_injective} that the LSD of the Schur-Hadamard product sequence $\{n^{-1/2}U_n\odot V_n\}$ where $L_U(i,j)=a^{i\wedge j}b^{i\vee j}$ and $L_V(i,j)=(i-j)^2$ is the semi-circular law.
\end{exm}
\begin{remark}
If we drop the assumption of injectivity, all we can say is that $\Pi^*_{X}(w)\cap \Pi^*_{Y}(w')\subseteq \Pi^*_U(w)\cap \Pi^*_V(w')$. So, if the LSD of $\{n^{-1/2}U_n\odot V_n\}$ exists, its moments will dominate the moments of the LSD of $\{n^{-1/2}X_n\odot Y_n\}$.
\end{remark}
\section{Toeplitz and Hankel}\label{sec:toephank}
Theorem~\ref{thm:gen} does not cover the situation where $L_Y$ is not a function of $L_X$, e.g., the case where $X_n$ is Toeplitz and $Y_n$ is Hankel (see the second row of Table~\ref{table2}). In order to proceed further we need the concept of \emph{Catalan} words from \citet{bose2008another}.

A Catalan word of length $2$ is just a double letter $aa$. In general, a Catalan word of length $2k$, $k>1$, is a word $w\in \mathcal{W}_{2k}$ containing a double letter such that if one deletes the double letter the reduced word becomes a Catalan word of length $2k-2$.
For example, $abba$, $aabbcc$, $abccbdda$ are Catalan words whereas $abab$, $abccab$, $abcddcab$
are not. The set of all Catalan word of length $2k$ will be denoted by $\mathcal{C}_{2k}$. 
There is a bijection between Catalan words and non-crossing pair partitions of the set $\{1,2,\cdots,2k\}$ whence it follows that
\begin{equation}
\#\mathcal{C}_{2k}=\frac{1}{k+1}\binom{2k}{k},
\end{equation}
the ubiquitous Catalan number from combinatorics.

 By the theory developed in Section~\ref{sec:shproduct}, it suffices to compute $p_Z(w,w')$ for different combination of word pairs $(w,w')\in \mathcal{W}_{2k}^2$. Note that $\pi \in \Pi_X^*(w)\cap \Pi_Y^*(w')$ means that we have exactly $2k$ constraints on the vertices $\pi(i)$, with each word giving rise to $k$ constraints. To elaborate, each $w$-match $(i,j)$ gives rise to the restriction
\[
L_X(\pi(i-1),\pi(i))=L_X(\pi(j-1),\pi(j)),
\]
and each $w'$-match $(k,l)$ gives rise to the restriction
\[
L_Y(\pi(k-1),\pi(k))=L_Y(\pi(l-1),\pi(l)).
\]
Thus we expect that if $w\neq w'$ and the two links functions $L_X$ and $L_Y$ behave nicely, then we will have more than $k$ independent constraints so that $\#\Pi_X^*(w)\cap \Pi_Y^*(w')=O(n^k)$ and a fortiori $p_Z(w,w')=0$. We shall call two link functions $L_X$ and $L_Y$, which satisfy Property B, \emph{compatible} if, for $w\neq w'$, we have $p_Z(w,w')=0$. We shall also write $(L_X,L_Y)\rightsquigarrow L_W$ if 
\[
p_Z(w,w)=\begin{cases}
1, & \text{if }w\in \mathcal{C}_{2k}\\
0, & \text{otherwise.}
\end{cases}
\]
\begin{prop}\label{thm:semi_circlular}
Suppose $L_X$ and $L_Y$ are compatible and $(L_X,L_Y)\rightsquigarrow L_W$. Then the LSD of $\{n^{-1/2}X_n\odot Y_n\}$ is the semi-circular law.
\end{prop}
\begin{proof}
We have
\begin{align*}
\beta_{2k}^Z &=\sum_{(w,w')\in \mathcal{W}_{2k}^2}p_Z(w,w')\\
&= \sum_{w\in\mathcal{C}_{2k}}p_Z(w,w)=\#\mathcal{C}_{2k}.
\end{align*}
This completes the proof because the $2k$-th moment of the semi-circular law is $\#\mathcal{C}_{2k}$.
\end{proof}
We write $(L_X,L_Y)\Rightarrow L_W$ if $L_X$ and $L_Y$ together determine the Wigner link function $L_W$ in the sense that $L_X(i,j)=L_X(k,l)$ and $L_Y(i,j)=L_Y(k,l)$ together imply that $L_W(i,j)=L_W(k,l)$. 
\begin{lemma}\label{lem:gen_lemma}
If $(L_X,L_Y)\Rightarrow L_W$, then
$(L_X,L_Y)\rightsquigarrow L_W$.
\end{lemma}
\begin{proof}
Suppose that $\pi \in \Pi_X^*(w)\cap \Pi_Y^*(w)$. Then
\begin{align*}
& w[i]=w[j]\\
\Rightarrow & L_X(\pi(i-1),\pi(i))=L_X(\pi(j-1),\pi(j)) \text{ and }  L_Y(\pi(i-1),\pi(i))=L_Y(\pi(j-1),\pi(j))\\
\Rightarrow & L_W(\pi(i-1),\pi(i))=L_W(\pi(j-1),\pi(j)).
\end{align*}
Therefore $\pi\in\Pi_W^*(w)$ and as a consequence $\Pi_X^*(w)\cap \Pi_Y^*(w)\subseteq \Pi_W^*(w)$. The other inclusion is always true because $L_X$ and $L_Y$ are symmetric link functions so that $\Pi_W^*(w)\subseteq \Pi_X^*(w)$ and $\Pi_W^*(w)\subseteq \Pi_Y^*(w)$. Therefore $\Pi_X^*(w)\cap \Pi_Y^*(w)= \Pi_W^*(w)$ which means that we have
\[
p_Z(w,w)=p_W(w),
\]
for each $w\in\mathcal{W}_{2k}$. As for the Wigner matrix
\[
p_W(w)=\begin{cases}
1, & \text{if }w\in \mathcal{C}_{2k}\\
0, & \text{otherwise,}
\end{cases}
\]
the proof is now complete.
\end{proof}

We shall establish below that the Schur-Hadamard product of Toeplitz and Hankel has the semi-circular LSD by verifying the conditions of Proposition~\ref{thm:semi_circlular}. See Figure~\ref{fig:toephank}.
\begin{figure}[t!]
 \centering
 \includegraphics{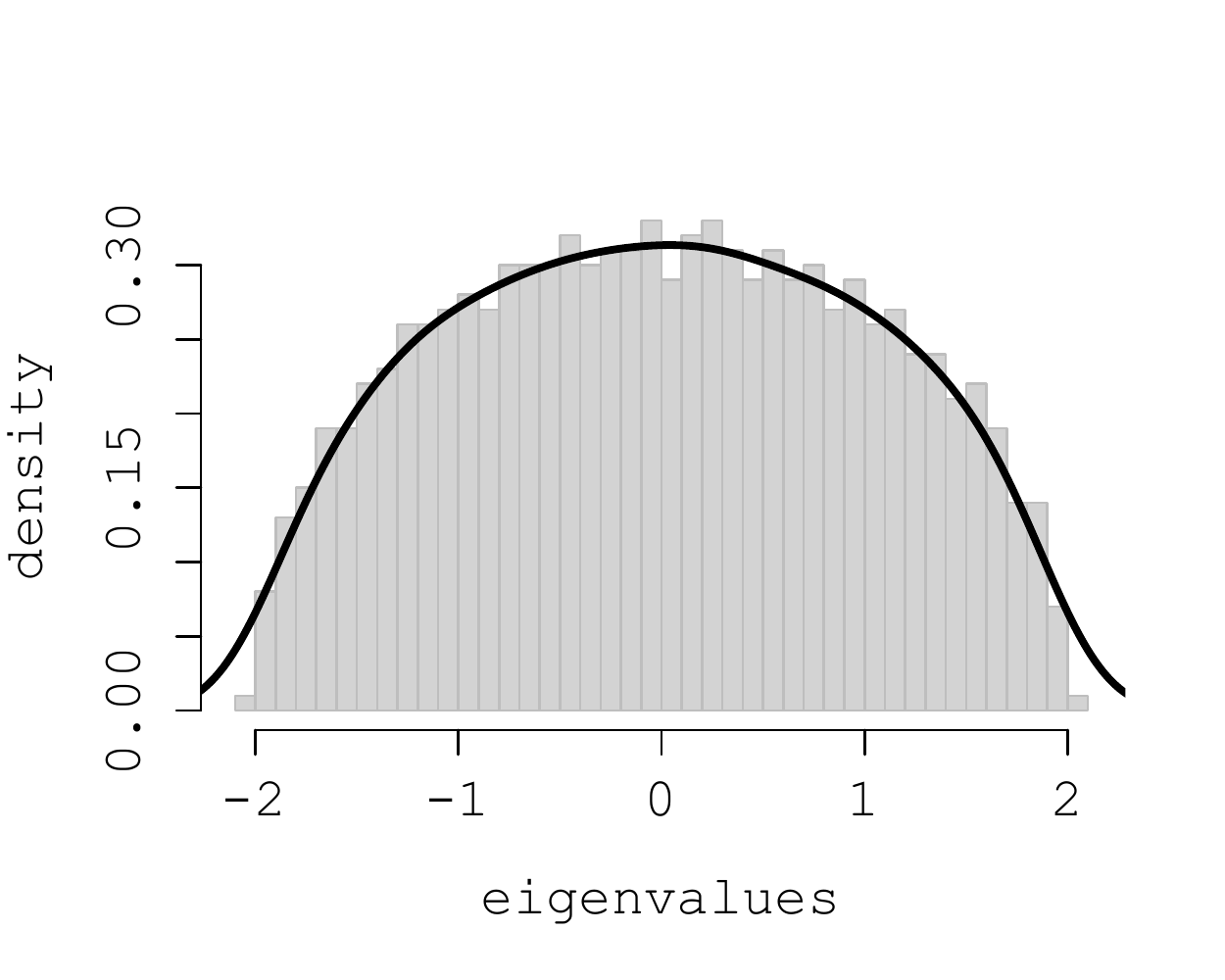}
 \caption{\footnotesize{Histogram and kernel density estimate for the ESD of $n^{-1/2}X_n \odot Y_n$, where $n=1000$, $X_n$ is a random Toeplitz matrix, $Y_n$ is a random Hankel matrix, they are independent and have $\mathcal{N}(0,1)$ entries.}}
 \label{fig:toephank}
 \end{figure}
 
We first make a simplification. Let $s(i)=\pi(i)-\pi(i-1)$. Define
 \[
 \Pi'(w)=\{\pi \mid w[i]=w[j]\Rightarrow s(i)+s(j)=0\}.
 \] 
 From \citet{bose2008another} it is known that for the Toeplitz matrix,
 \[
 p(w)=\lim_n \frac{1}{n^{1+k}}\#\Pi'(w).
 \]
 Therefore, we need only look at $\Pi_X'(w) \cap \Pi_Y^*(w')$ and if the limit exists we have
 \[
 p_Z(w,w')=\lim_n \frac{1}{n^{1+k}}\#\Pi_X'(w) \cap \Pi_Y^*(w').
 \]
 We shall use this in the following lemma.
\begin{lemma}\label{lem:toep_hank_non_eq}
$L_T$ and $L_H$ are compatible.
\end{lemma}
\begin{proof}
We shall show that if $w\neq w'$, then $\#\Pi_X'(w)\cap \Pi_Y^*(w')=O(n^k)$.  This would imply that 
\[
p_Z(w,w')=0, \text{ if $w\neq w'$}.
\]
To do this it is enough to show that in addition to the $k$ constraints on the choices of $\pi$ arising from the Toeplitz link function (or the Hankel link function), there is at least one more additional constraint. Note that for Toeplitz or Hankel link functions, the natural constraints arising from matches enable one to express each non-generating vertex $\pi(j)$ as a linear combination of the generating vertices $\pi(i)$ preceding it (i.e., $i<j$). We shall show that if we combine the $2k$ constraints corresponding to $(w,w')$, then we can write some generating vertex $\pi(i)$ as a linear combination of the preceding generating vertices. This will be the extra constraint we are seeking. 

Consider the positions where new letters (letters appearing for the first time) appear. Note that the positions of the new letters fix their pattern and also fix the position of the old letters but not their pattern. Suppose the positions of the new letters are not all same in $w$ and $w'$. Let $i$ be the \textit{first} place where a new letter appears in (say) $w$ but an old letter appears in $w'$. So $\pi(i)$ is a generating vertex for $\Pi_T(w)$, but for $\Pi_H(w')$ it is non-generating, so that we can express $\pi(i)$ as a linear combination of generating vertices $\pi(k)$, $k<i$, $k\in S_H$. but note that prior to $i$ the generating vertices in $S_T$ and $S_H$ are same (indeed $i$ is the first position where there is a difference). Therefore, we can express $\pi(i)$ as a linear combination of generating vertices $\pi(k)$, $k<i$, $k\in S_T$, which is an extra constraint. 

Now suppose that all the new letters appear at the same positions in $w$ and $w'$ so that $S_T=S_H=S$, say. Now, since $w\neq w'$, there exists $j$ such that there are $i,i^*\in S$ both less than $j$ with $i\sim^T j$ and $i^*\sim^H j$. Assume that $j$ is the first letter of this type. Further without loss of generality we may assume that $i^*<i$. Then we have the following two constraints:
\begin{align}
\pi(i)-\pi(i-1)&=\pi(j-1)-\pi(j) \text{, and}\\
\pi(i^*)+\pi(i^*-1)&=\pi(j)+\pi(j-1).
\end{align}
Eliminating $\pi(j)$ from these two constraints we arrive at 
\begin{equation}\label{cons:1}
2\pi(j-1)=\pi(i)-\pi(i-1)+\pi(i^*)+\pi(i^*-1).
\end{equation}
\textbf{Case I: ($i=j-1$).} In this case (\ref{cons:1}) becomes
\[
\pi(i)=-\pi(i-1)+\pi(i^*)+\pi(i^*-1),
\]
which is an additional constraint because we are being able to express the generating vertex $\pi(i)$ as a linear combination of generating vertices $\pi(k)$ with $k<i$.\\\\
\textbf{Case II: ($i<j-1$).} If $\pi(j-1)$ is a generating vertex, then again we have an extra constraint because via (\ref{cons:1}) we are able to express the generating vertex $\pi(j-1)$ as a linear combination of generating vertices $\pi(k)$ with $k<j-1$. So suppose that $\pi(j-1)$ is non-generating. Then there exists $i_1\in S$ such that $i_1 \sim^T (j-1)$ and $i_1 \sim^H (j-1)$ (recall that $j$ is assumed to be the first index where the words $w$ and $w'$ differ). This implies that we have
\begin{align*}
\pi(i_1)-\pi(i_1-1) & =\pi(j-2)-\pi(j-1) \text{, and}\\
\pi(i_1)+\pi(i_1-1) & =\pi(j-2)+\pi(j-1),
\end{align*}
which simplify to 
\begin{align*}
\pi(i_1)& =\pi(j-2) \text{, and}\\
\pi(i_1-1) & =\pi(j-1).
\end{align*}
Thus (\ref{cons:1}) becomes
\begin{equation}\label{cons:2}
2\pi(i_1-1)=\pi(i)-\pi(i-1)+\pi(i^*)+\pi(i^*-1).
\end{equation}
If $i_1-1\leqslant i$, using (\ref{cons:2}) we can express $\pi(i)$ as a linear combination of generating vertices $\pi(k)$ with $k<i$ thus giving rise to an extra constraint. On the other hand if $i_1-1>i$ and $i_1-1\in S$ then (\ref{cons:2}) gives an extra constraint where $\pi(i_1-1)$ is expressed as a linear combination of generating vertices $\pi(k)$ with $k<i_1-1$. Finally, if $i_1-1>i$ and $i_1-1\notin S$, then, by the same argument as above with the role of $j-1$ being played by $i_1-1$, we can find $i_2\in S$ such that $i_2 \sim^T (i_1-1)$ and $i_2 \sim^H (i_2-1)$ and so on. It is clear that if we continue this procedure, then at some point we will obtain $i_m\in S$, $m\geqslant 1$ such that $i_m \sim^T (i_{m-1}-1)$ and $i_m \sim^H (i_{m-1}-1)$ and $i_m-1\leqslant i$ and so we will be able to express $\pi(i)$ as a linear combination of generating vertices $\pi(k)$ with $k<i$, thus obtaining an extra constraint.
\end{proof}
\begin{lemma}\label{lem:toep_hank_eq}
$(L_T,L_H)\rightsquigarrow L_W$.
\end{lemma}
\begin{proof}
As $(L_T,L_H)\Rightarrow L_W$, Lemma~\ref{lem:gen_lemma} directly applies. However, we give here an alternate argument that applies to some cases where $(L_X,L_Y)\nRightarrow L_W$ (for example, note that $(L_T,L_{RC})\nRightarrow L_W$ but one can show that $(L_T,L_{RC})\rightsquigarrow L_W$, see Remark~\ref{rem:others} below).
 
Fix $w\in\mathcal{C}_{2k}$. Suppose that we have a double letter at position $i$, i.e., $w[i]=w[i+1]$. Suppose $\pi \in \Pi_Y^*(w)$. Then we have 
\[
\pi(i-1)+\pi(i)=\pi(i)+\pi(i+1).
\]
So $\pi(i-1)=\pi(i+1)$. This implies that
\[
s(i)=\pi(i)-\pi(i-1)=\pi(i)-\pi(i+1)=-s(i+1).
\]
Now, deleting this double letter, i.e., identifying $\pi(i-1)$ and $\pi(i)$, we are left with a Catalan word $\hat{w}$ of length $2k-2$ and the reduced circuit $\hat{\pi}\in \Pi_Y^*(\hat{w})$. It has a double letter and we can repeatedly use the above argument until the whole word is emptied. What this argument gives is this: $w[i]=w[j]\Rightarrow s(i)+s(j)=0$. So $\pi\in \Pi_X'(w)$. Therefore $\Pi_Y^*(w)\subseteq \Pi_X'(w)$. So $\Pi_X'(w) \cap \Pi_Y^*(w)=\Pi_Y^*(w)$ and therefore $p_Z(w,w)$ exists and equals $p_Y(w)=1$.

Now fix $w\in \mathcal{W}_{2k}$. Consider the match $(i,j)$ where $j$ is the position of the first old letter. We have 
\begin{align*}
\pi(i)-\pi(i-1) & =\pi(j-1)-\pi(j) \text{, and}\\
\pi(i)+\pi(i-1) & =\pi(j-1)+\pi(j),
\end{align*}
which simplify to 
\begin{align}
\pi(i)& =\pi(j-1) \text{, and}\label{cons:3}\\
\pi(i-1) & =\pi(j)\label{cons:4}.
\end{align}
But by definition of $j$, $\pi(j-1)$ is generating, so (\ref{cons:3}) is a new constraint unless $i=j-1$, in which case $(i,j)$ is a double letter and (\ref{cons:4}) becomes $\pi(i-1)=\pi(i+1)$, which is an automatic constraint in both $\Pi_X'(w)$ and $\Pi_Y^*(w)$. Delete this double letter and apply the above argument on the reduced word. Clearly, if $w$ is non-Catalan, at some point we will be left with a non-empty word with no double letters, thus getting an extra constraint and therefore we will have $p_Z(w,w)=0$. 
\end{proof}
\begin{remark}\label{rem:others}
If we take $X_n$ as Symmetric Circulant, then Lemmas \ref{lem:toep_hank_non_eq} and \ref{lem:toep_hank_eq} continue to hold. Indeed, if for a pair matched word $w$ we define
\[
\Pi_X'(w)=\{\pi\mid w[i]=w[j]\Rightarrow s(i)+s(j)=0,\pm n\},
\]
then, from \citet{bose2008another}, we know that
\[
p_X(w)=\lim_n \frac{1}{n^{1+k}}\#\Pi_X'(w)=1.
\]
One can readily see that the proof of Lemmas \ref{lem:toep_hank_non_eq} and \ref{lem:toep_hank_eq} goes through in this case with minor modifications. Similarly, we can take $Y_n$ as Reverse Circulant or Doubly Symmetric Hankel. Therefore, the conclusion of Theorem~\ref{thm:semi_circlular} hold in these cases as well and thus we completely  obtain the second row of Table~\ref{table2}.
\end{remark}

\begin{appendices}
\section{Two counting lemmas}\label{sec:twolemma}
Recall that $K^L_{h,t}$ is the set of $t$-tuples of circuits $(\pi_1,\cdots,\pi_t)$ of length $h$ such that they are jointly and across $L$-matched. 
If $L$ satisfies Property $B$, then Lemma 2(a) of \citet{bose2008another} says that
\begin{equation}\label{eq:t=4}
\#K^L_{h,4}=O(n^{2h+2}).
\end{equation}
The arguments of \citet{bose2008another} are adaptations of those of \citet{bryc2006spectral} who proved this estimate for Toeplitz and Hankel matrices. One can modify these arguments to accommodate other values of $t$. For the reader's convenience we provide a proof for $t=2$ which we have used in the proof of Lemma~\ref{lem:inprob} and will be using in the proof of Lemma~\ref{lem:as}. We also state the version for $t=3$ without proof as it will be needed while proving Lemma~\ref{lem:as}.
\begin{lemma}\label{lem:count1}
If $L$ satisfies Property $B$, then 
\begin{align*}
\#K^L_{h,2}&=O(n^{h+1}) \text{, and}\\
\#K^L_{h,3}&=O(n^{\lfloor\frac{3h}{2}\rfloor+2}).
\end{align*}
\end{lemma}
\begin{proof}
 Consider all circuits ($\pi_1$, $\pi_2$) of length $h$ which are jointly $L$-matched and across $L$-matched. Consider all possible edges $(\pi_j(i-1),\pi_j(i))$, $1\leqslant j \leqslant 2$ and $1 \leqslant i \leqslant h$. Since the circuits are jointly and across $L$-matched, there are at most $h$ distinct $L$-values in these $2h$ edges.

Note that the number of partitions of the $2h$ edges into distinct groups of $L$-matching edges, with at least two edges in each group, is independent of $n$. So, for a fixed integer $1\leqslant u \leqslant h$, it is enough to establish the required estimate for the number of pairs of circuits for which there are exactly $u$ distinct $L$-values.

First assume that $1\leqslant u \leqslant h-1$. We count the total number of choices in the following way:

\begin{enumerate}
\item The generating vertices $\pi_1(0), \pi_2(0)$ may be chosen in total $n^2$ many ways.

\item  Now arrange the values $L(\pi_j(i-1),\pi_j(i))$, $1\leqslant j \leqslant 2$, $1 \leqslant i \leqslant h$ from left to right, starting with $\pi_1$ followed by $\pi_2$. Then the generating vertices $\pi_j(i)$, for which $L(\pi_j(i-1),\pi_j(i))$ is the first one of the distinct $L$-values in this sequence, have at most $n^u$ choices.

\item Having chosen these vertices, using Property B and $L$-matchings, the rest of the vertices in all the circuits may be chosen from left to right in at most $(\Delta_L) ^{2h-u-2}$ ways.  
\end{enumerate}

Now, since $u\leqslant h-1$, the total number of choices is bounded by
\[
n^2n^u(\Delta_L) ^{2h-u-2}=O(n^{u+2})= O(n^{h+1}).
\]

Now consider the case $u=h$. Then each $L$-value is shared by exactly two edges. Now we seek to identify one generating vertex that has only finitely many choices.

By reordering the two circuits if necessary, we have an $L$-value that is assigned, as the first and only one, to exactly one edge, say $(\pi_1(i-1),\pi_1(i))$ of $\pi_1$. Pick this $L$-value. The rest of the $(u-1)$ generating vertices may be chosen in at most $n^{u-1}=n^{h-1}$ ways. By the following dynamic construction of $\pi_1$ we show that $\pi_1(i)$ can have only finitely many choices:

Start with $\pi_1(0)$ and choose $\pi_1(j)$ till $j\leqslant i-1$, honouring the $L$-matches. Now start from the tail end of $\pi_1$, i.e., from $\pi_1(h)=\pi_1(0)$ and choose the vertices $\pi_1(j)$ in a right-to-left manner. When $\pi_1(i+1)$ is chosen, since the $L$-value $L(\pi_1(i),\pi_1(i+1))$ appears elsewhere, note that $\pi_1(i)$ can have only finitely many choices. Thus the total number of choices is bounded by
\[
n^2n^{u-1}(\Delta_L)^{2h-u-2+1}= O(n^{u+1})= O(n^{h+1}).
\]
This completes the proof in the case $t=2$. The proof in the case $t=3$ is an easy modification of the argument above and hence omitted.
\end{proof}
The following lemma will be repeatedly used in the verification of Condition $(ii')$ of Lemma~\ref{lem:main} in Lemma~\ref{lem:as} of Appendix~\ref{sec:asconv}. 
\begin{lemma}\label{lem:count2}
Consider two $h$-circuits $\pi_1$ and $\pi_2$. Suppose $\pi_1$ is pair-matched with respect to $L_X$ (which necessitates that $h$ be even) and shares no $L_X$ values with $\pi_2$. Also suppose that $\pi_1$ and $\pi_2$ share an $L_Y$ value. Then, contingent on the event that $\pi_2$ has been already chosen, one can choose $\pi_1$ in $O(n^{h/2})$ ways, honouring the stated constraints.
\end{lemma}
\begin{proof}
Consider $\pi_1$ from left to right. There is one and hence a first index $i$ such that the $L_Y$-value $L_Y(\pi_1(i-1),\pi_1(i))$ appears in $\pi_2$. Now consider the $L_X$-matches on $\pi_1$. Since $\pi_1$ is pair-matched, it has $h/2+1$ generating vertices and therefore in absence of any further constraints one can choose $\pi_1$ in $O(n^{h/2+1})$ ways. We shall show that under the setup of the lemma one among these generating vertices has only finitely  many choices. Note that we may assume without loss of generality that $\pi_1(i)$ is a generating vertex with respect to $L_X$ (indeed, otherwise we may start filling the circuit from right to left and define generating vertices according to that order to ensure that $\pi_1(i)$ is generating). But now, since the value $L_Y(\pi_1(i-1),\pi_1(i))$ is fixed, after choosing $\pi_1(0), \cdots , \pi_1(i-1)$ with respect to $L_X$, there are only finitely many choices left for (the generating vertex) $\pi_1(i)$. This completes the proof.
\end{proof}
\section{Almost sure weak convergence}\label{sec:asconv}
The proof of almost sure weak convergence is presented in the following lemma.
\begin{lemma}\label{lem:as} 
Suppose $L_X$ and $L_Y$ satisfy Property B and the input sequences satisfy Assumption (A1). Then  $\{\beta_h(n^{-1/2}Z_n)\}$ satisfies Condition $(ii')$ of Lemma~\ref{lem:main} for any $h$. 
\end{lemma}
\begin{proof}
Using (\ref{eq:t=4}) in our context we have
\begin{align}\label{bound:quadruples}
\#(K^{L_X}_{h,4}\cup K^{L_Y}_{h,4}) & \leq  \#K^{L_X}_{h,4} + \#K^{L_Y}_{h,4} \\
\nonumber									& =  O(n^{2h+2}) +O(n^{2h+2}) \\
\nonumber									& = O(n^{2h+2}).
\end{align}
Now write
\begin{align}\label{repr:fourth}
\E[\beta_h(n^{-1/2}Z_n)-\E(\beta_h(n^{-1/2}Z_n))]^4
&= \E[n^{-1}\tr(n^{-1/2}Z_n)^h-\E(n^{-1}\tr(n^{-1/2}Z_n)^h)]^4\\
\nonumber&= \frac{1}{n^{2h+4}}\E[\tr Z_n^h-\E\tr Z_n^h]^4\\
\nonumber&=\frac{1}{n^{2h+4}}\sum_{(\pi_1,\pi_2,\pi_3,\pi_4)}\E(\prod_{j=1}^{4}z_{\pi_j}-\E z_{\pi_j}).
\end{align}
Therefore, using decomposition (\ref{decomp}) we have
\[
\prod_{j=1}^{4}(z_{\pi_j}-\E z_{\pi_j})  = \prod_{j=1}^{4}( (x_{\pi_j}-\E x_{\pi_j})(y_{\pi_j}-\E y_{\pi_j})+(y_{\pi_j}-\E y_{\pi_j})\E x_{\pi_j}+(x_{\pi_j}-\E x_{\pi_j})\E y_{\pi_j}).
\]
If $(\pi_1 , \pi_2 , \pi_3 , \pi_4)$ are not jointly $L_X$-matched, then one of the circuits, say $\pi_k$, has an
$L_X$-value which does not occur anywhere else. Therefore $\E x_{\pi_k}=0$. So
\[
z_{\pi_k}-\E z_{\pi_k}  = x_{\pi_k}y_{\pi_k},
\]
and
\begin{equation}\label{repr:1}
\prod_{j=1}^{4}(z_{\pi_j}-\E z_{\pi_j})  = x_{\pi_k}y_{\pi_k} \prod_{\begin{subarray}{1}
j=1\\
j\neq k
\end{subarray}}^{4} (x_{\pi_j}y_{\pi_j}-\E x_{\pi_j}\E y_{\pi_j}).
\end{equation}
Because of the independence of $X_n$ and $Y_n$ and of the input sequences we can conclude from this representation that
\[
\E\prod_{j=1}^{4}(z_{\pi_j}-\E z_{\pi_j})=0,
\]
since the input $X$-variable corresponding to the single $L_X$-value appears in the product (\ref{repr:1}) inside $x_{\pi_k}$ and is independent of every other term in the product. Therefore, in order to have a non-zero contribution, $(\pi_1 , \pi_2 , \pi_3 , \pi_4)$ have to be jointly $L_X$-matched and by the same argument jointly $L_Y$-matched.

Now suppose that $(\pi_1 , \pi_2 , \pi_3 , \pi_4)$ are jointly $L_X$ as well as $L_Y$-matched but neither across $L_X$-matched nor across $L_Y$-matched. Then there is a circuit, say $\pi_k$, which is only \emph{self $L_X$-matched}, i.e., none of its $L_X$-values is shared with those of the other circuits. Similarly, there is a circuit $\pi_l$ that is only self $L_Y$-matched. Now note that $(x_{\pi_k}-\E x_{\pi_k})$ is independent of $(x_{\pi_j}-\E x_{\pi_j})$ for $j\neq k$ and similarly $(y_{\pi_l}-\E y_{\pi_l})$ is independent of $(y_{\pi_j}-\E y_{\pi_j})$ for $j\neq l$. If $k=l$ (which is always the case in the setup of Theorem~\ref{thm:gen}), using these facts along with the independence of $X_n$ and $Y_n$ and the decomposition (\ref{decomp}) we can write
\begin{align*}
\E\prod_{j=1}^{4}(z_{\pi_j}-\E z_{\pi_j})= \E (x_{\pi_k}-\E x_{\pi_k}) & \E(y_{\pi_k})\E( \prod_{\begin{subarray}{1}
j=1\\
j\neq k
\end{subarray}}^{4} (z_{\pi_j}-\E z_{\pi_j}))\\& + \E (y_{\pi_k}-\E y_{\pi_k}) \E(x_{\pi_k})\E( \prod_{\begin{subarray}{1}
j=1\\
j\neq k
\end{subarray}}^{4} (z_{\pi_j}-\E z_{\pi_j}))\\
&=0.
\end{align*}
If $k\neq l$ then $\E\prod_{j=1}^{4}(z_{\pi_j}-\E z_{\pi_j})$ is not necessarily 0. However, since, by Assumption (A1), $\E\prod_{j=1}^{4}(z_{\pi_j}-\E z_{\pi_j})$ is bounded uniformly across all possible quadruples, it suffices to prove an $O(n^{2h+3-\delta})$ estimate, $\delta>0$, on the number of quadruples of circuits in the $k\neq l$ case. We shall prove such estimates (and we will not try to be optimal) in each of the following three possible cases:

\textbf{Case I}. $\pi_1$ is self $L_X$-matched, $(\pi_2, \pi_3, \pi_4)$ are across $L_X$-matched and  $\pi_2$ is self $L_Y$-matched, $(\pi_1, \pi_3, \pi_4)$ are across $L_Y$-matched. By Lemma~\ref{lem:count1}, if we just consider the $L_X$-matches, then $(\pi_2, \pi_3, \pi_4)$ can be chosen together in at most $O(n^{\lfloor 3h/2 \rfloor +2})$ many ways. So, if $\pi_1$ has at least one edge of order $\geqslant 3$, then, by (\ref{eq:three_or_more}), we can choose $\pi_1$ in $O(n^{\lfloor (h+1)/2\rfloor})$ ways. Thus, in this case, the total number choices for the quadruples $(\pi_1,\pi_2, \pi_3, \pi_4)$ is 
\[
\underbrace{O(n^{\lfloor \frac{3h}{2} \rfloor +2})}_{(\pi_2, \pi_3, \pi_4)}\underbrace{O(n^{\lfloor \frac{h+1}{2}\rfloor})}_{\pi_1}=O(n^{2h+\frac{5}{2}}).
\]
So we may assume that $\pi_1$ is pair-matched with respect to $L_X$ and by the same token $\pi_2$ is pair-matched with respect to $L_Y$.

Choose $(\pi_2, \pi_3, \pi_4)$ honouring the $L_X$-matches in  $O(n^{\lfloor 3h/2 \rfloor +2})$ ways. Now $\pi_1$ shares an $L_Y$ value either with $\pi_3$ or $\pi_4$, since $(\pi_1, \pi_3, \pi_4)$ are across $L_Y$-matched. Therefore, by Lemma~\ref{lem:count2} we can choose $\pi_1$ in $O(n^{h/2})$ ways. Therefore, the total number of choices for the quadruples $(\pi_1, \pi_2, \pi_3, \pi_4)$ is 
\[
\underbrace{O(n^{\lfloor 3h/2 \rfloor +2})}_{(\pi_2, \pi_3, \pi_4)}\underbrace{O(n^{\frac{h}{2}})}_{\pi_1}=O(n^{2h+2}).
\]

\textbf{Case II}. $\pi_1$, $\pi_2$ are self $L_X$-matched and $(\pi_3,\pi_4)$ are across $L_X$-matched while $\pi_3$ is self $L_Y$-matched and $(\pi_1,\pi_2,\pi_4)$ are across $L_Y$-matched. Note that we may again assume that both $\pi_1$ and $\pi_2$ are pair-matched with respect to $L_X$, because otherwise upon choosing $(\pi_3, \pi_4)$, honouring the $L_X$-matches, in $O(n^{h+1})$ ways (by Lemma~\ref{lem:count1}), we can choose both $\pi_1$ and $\pi_2$ in $O(n^{\lfloor \frac{h+1}{2}\rfloor})$ ways with respect to $L_X$, so that the total number of choices becomes
\[
\underbrace{O(n^{h+1})}_{(\pi_3,\pi_4)}\underbrace{O(n^{\lfloor \frac{h+1}{2}\rfloor})}_{\pi_2}\underbrace{O(n^{\lfloor \frac{h+1}{2}\rfloor})}_{\pi_1}=O(n^{2h+2}).
\]
 Choose $(\pi_3, \pi_4)$, honouring the $L_X$-matches, in $O(n^{h+1})$ ways. Now, since $(\pi_1,\pi_2,\pi_4)$ are across $L_Y$-matched, two possibilities might arise:
\begin{enumerate}
\item $\pi_1$ and $\pi_2$ both share an $L_Y$-value with $\pi_4$.
\item $\pi_1$, $\pi_2$ share an $L_Y$-value and $\pi_2$, $\pi_4$ share an $L_Y$-value.
\end{enumerate}
In the first case, since we have already chosen $\pi_4$, by Lemma~\ref{lem:count2} $\pi_1$ and $\pi_2$ both can be chosen in $O(n^{h/2})$ ways so that the total number of choices for the quadruples $(\pi_1, \pi_2, \pi_3, \pi_4)$ is 
\[
\underbrace{O(n^{h+1})}_{(\pi_3,\pi_4)}\underbrace{O(n^{\frac{h}{2}})}_{\pi_2}\underbrace{O(n^{\frac{h}{2}})}_{\pi_1}=O(n^{2h+1}).
\]

In the second case, again by Lemma~\ref{lem:count2}, we can choose $\pi_2$ in $O(n^{h/2})$ ways and thereafter $\pi_1$ in $O(n^{h/2})$ ways so that the total number of choices again becomes 
\[
\underbrace{O(n^{h+1})}_{(\pi_3,\pi_4)}\underbrace{O(n^{\frac{h}{2}})}_{\pi_2}\underbrace{O(n^{\frac{h}{2}})}_{\pi_1}=O(n^{2h+1}).
\]

\textbf{Case III}. $\pi_1$, $\pi_2$ are self $L_X$-matched and $(\pi_3,\pi_4)$ are across $L_X$-matched while $\pi_3$, $\pi_4$ are self $L_Y$-matched and $(\pi_1,\pi_2)$ are across $L_Y$-matched. Once again we may and will assume that both $\pi_1$ and $\pi_2$ are pair-matched with respect to $L_X$. Choose $(\pi_3,\pi_4)$ honouring the $L_X$ constraints in $O(n^{h+1})$ ways. Now choose $\pi_2$ in $O(n^{h/2+1})$ ways honouring the $L_X$ constraints. Since $(\pi_1,\pi_2)$ are across $L_Y$-matched and $\pi_2$ has been chosen, by Lemma~\ref{lem:count2} we can choose $\pi_1$ in $O(n^{h/2})$ ways. Thus, in this case, the total number of choices for the quadruples $(\pi_1, \pi_2, \pi_3, \pi_4)$ is 
\[
\underbrace{O(n^{h+1})}_{(\pi_3,\pi_4)}\underbrace{O(n^{\frac{h}{2}+1})}_{\pi_2}\underbrace{O(n^{\frac{h}{2}})}_{\pi_1}=O(n^{2h+2}).
\]

All the other types of quadruples of circuits are contained in  $K^{L_X}_{h,4}\cup K^{L_Y}_{h,4}$. Therefore, by what have been established so far, we conclude that
\[
\E[\beta_h(n^{-1/2}Z_n)-\E(\beta_h(n^{-1/2}Z_n))]^4=O(n^{-(1+\delta)}),
\]
for some suitable $\delta>0$, which completes the verification of Condition $(ii')$ of Lemma~\ref{lem:main}.
\end{proof}
\end{appendices}
\section{Acknowledgements}
We thank the anonymous referees for their comments and for pointing us to important literature that we had missed.

\bibliographystyle{chicago}
\bibliography{mybib}

\end{document}